\documentclass[a4paper,12pt]{elsarticle}
\usepackage{amsmath}
\usepackage{cases}
\usepackage{mathrsfs}
\usepackage{bbm}
\usepackage{amssymb}
\usepackage{amscd}
\usepackage{amsfonts,latexsym,amsmath,amsthm,amsxtra,mathdots,amssymb,latexsym,mathabx}
\usepackage[all,cmtip]{xy}
\RequirePackage{amsmath} \RequirePackage{amssymb}
\usepackage{color}
\usepackage{colordvi}
\usepackage{multicol}
\usepackage{hyperref}
\usepackage{mathtools}
\usepackage[margin=1in]{geometry}
\usepackage{xcolor}
\hypersetup{
	colorlinks,
	linkcolor={red!50!black},
	citecolor={blue!50!black},
	urlcolor={blue!80!black}
}

\newcommand{\sumthree}{\operatorname*{\sum\sum\sum}}

\numberwithin{equation}{section}
\newtheorem{thm}{Theorem}[section]
\newtheorem{lem}[thm]{Lemma}
\newtheorem{cor}[thm]{Corollary}
\newtheorem {rem}[thm]{Remark}
\allowdisplaybreaks[4]

\title{Mixed integral moments of the Hecke $L$-functions and Riemann zeta function}
\author{Zhaoyan Chen} %% Author name
\affiliation{organization={School of Mathematics},%Department and Organization
	addressline={Shandong University}, 
	city={Jinan Shandong},
	postcode={250100}, 
	country={China}}

\begin{document}


\begin{abstract}
In this paper, let $f$ be a Hecke cusp form for $SL(2,\mathbb{Z})$. We establish an asymptotic formula for the mixed moment of $\zeta^{2}(s)$ and $L(s,f)$ on the critical line, valid for both holomorphic and Maass forms.
\end{abstract}

\maketitle

\ead{zychen@mail.sdu.edu.cn}

	\section{Introduction}
	It is an important problem to understand the asymptotic formula of integral means of automorphic $L$-functions on the critical line. Classically, this line of research begins with the study of the moments of the Riemann zeta function, defined by
	$$
	\begin{aligned}
		I_{k}(T)=\int_{0}^{T}|\zeta(\frac{1}{2}+it)|^{2k} \mathrm{~d} t.
	\end{aligned}
	$$
    Keating and Snaith \cite{keating2000random} conjectured that the general form for the $k$-th moment is given by:
	$$
	\begin{aligned}
		I_{k}(T)=T \mathcal{P}_k(\log T)+O\left(T^{1 / 2+\varepsilon}\right),
	\end{aligned}
	$$
	where $\mathcal{P}_k$ is a polynomial of degree $k^2$. The above conjecture has been proved for $k=1$ \cite{ingham1928mean}. The current record for the error term in $I_1(T)$ was due to Bourgain and Watt \cite{bourgain2018decoupling} who showed that $O(T^{\frac{1515}{4816}+\varepsilon})$. 
	
	For $k=2$, The earliest result was obtained by Ingham \cite{ingham1928mean}, who established an estimate with an error term of size $O(T\log^{3} T)$. In 1979, Heath-Brown \cite{heath1979fourth} improved the error term to $O(T^{7/8+\varepsilon})$. Zavorotny\u{\i} \cite{MR1683661} subsequently improved the error term in $I_2(T)$ to $O(T^{2/3+\varepsilon})$. The best known result to date, as presented in \cite{ivic1995fourth2}, is the following asymptotic formula:
	\begin{equation}\label{Ivic}
		\begin{aligned}
			I_{2}(T)=T\mathcal{P}_2(\log T)+O(T^{2/ 3}(\log T)^8).
		\end{aligned}
	\end{equation}
	The error term $O(T^{1/2+\varepsilon})$ has been obtained in the case of a smooth weight function \cite{ivic1995fourth}. Despite many attempts, the asymptotic formulas for higher moments have not been completely solved. 
	
	Owing to the favorable analytic properties of holomorphic cusp forms, the integral means of Hecke $L$-functions have been investigated in a number of works. For instance, Potter \cite{potter1940mean} obtained an asymptotic formula for the second moment when $\operatorname{Re}(s)>1/2$. In 1982, Good \cite{good1982square} extended this result to the critical line using the spectral theory of automorphic forms on $GL(2)$. The following asymptotic holds:
	\begin{equation}\label{good}
		\begin{aligned}
			\int_0^T\left|L\left(\frac{1}{2}+i t,f\right)\right|^2 d t=2 c_{-1} T\left\{\log \frac{T}{2 \pi e}+c_0\right\}+O\left((T \log T)^{2 / 3}\right).
		\end{aligned}
	\end{equation}
	Meurman \cite{meurman1987order} subsequently generalized this result to the Maass case. Similarly, no effective asymptotic formula has been obtained for the case of higher moments. 
	
	Das and Khan \cite{das2015simultaneous} investigated mixed moments in the $q$-aspect. More recently, Blomer \textit{et al.} \cite{blomer2017moments} proved, for prime moduli $p$, the impressive asymptotic
    $$
    \begin{aligned}
    \begin{aligned}
    \frac{1}{p-2}\sum_{\substack{\chi(\bmod p) \\
    			\chi \text { primitive }}} L\left(\frac{1}{2}, f\otimes \chi\right) \overline{L\left(\frac{1}{2},\chi\right)}^2 
    	=\frac{L(1,f)^2}{\zeta(2)}+O_{f, \varepsilon}\left(p^{-1/ 68+\varepsilon}\right).
    \end{aligned}	
    \end{aligned}
    $$
    The error term was subsequently improved to $O(p^{-1/64+\varepsilon})$ by Shparlinski \cite{shparlinski2019sums}, and further to $O(p^{-1/22+\varepsilon})$ for holomorphic $f$ and $O(p^{-5/152+\varepsilon})$ for Maass forms by Khan and Zhang \cite{khan2023error}. Tang and Wu \cite{tang2025mixed} recently extended the result to general moduli, obtaining an error term comparable to the best bound \cite{khan2023error} in the prime modulus case.
    
    In a recent work, the author \cite{Chen2026Integral} studied a mixed second moment involving the Riemann zeta function and Hecke $L$-functions in the $t$-aspect. More precisely, she derived
    $$
    \begin{aligned}
    	\int_T^{2 T} L\left(\frac{1}{2}+i t, f\right)\left|\zeta\left(\frac{1}{2}-i t\right)\right|^2 \mathrm{~d} t=L(1, f) T \log T+c_f T+O\left(T^{2 / 3}(\log T)^{13 / 2}\right).
    \end{aligned}
    $$
    Building on the author’s previous work, the present paper further establishes an asymptotic formula for the corresponding mixed moment without the absolute value. We will prove the following results.
	
	\begin{thm}\label{t1} 
		Let $f$ be a Hecke cusp form (either holomorphic or Maass) for $SL(2,\mathbb{Z})$. Suppose $V$ is a smooth function supported on $[1,2]$ satisfying
		\[
		V^{(i)}(x) \ll \Delta^i,
		\]
		for $i\geq 0$ with $\Delta\ll T^{\varepsilon}$. For any $\varepsilon > 0$, we have that
		\begin{equation}\label{eqt1}
			\int_{\mathbb{R}} V\left(\frac{t}{T}\right) L\left(\frac{1}{2}+it, f\right) {\zeta\left(\frac{1}{2}-it\right)}^2 \,\mathrm{d}t
			= 2 cT \frac{L(1, f)^2}{\zeta(2)}+O\left(T^{1/2+\varepsilon}\right),
		\end{equation}
		where $c = \int_{\mathbb{R}} V(\xi)\mathrm{~d} \xi$.
	\end{thm}
    \begin{rem}
    The same method applies for $f$ either a holomorphic or a Maass form. In fact, the case where $f$ is holomorphic is simpler, so we will present the proof for the Maass case.
    \end{rem}
	Since the test function $V$ in the Theorem~\ref{t1} admits oscillatory behavior, we are able to derive a nontrivial asymptotic formula for the integral with a sharp cutoff. This leads to the following result:
	\begin{cor}\label{c3} We have
	$$
	\begin{aligned}
		\int_{T}^{2T}L\left(\frac{1}{2}+it, f\right){\zeta\left(\frac{1}{2}-it\right)}^{2}\mathrm{~d} t=2 T \frac{L(1, f)^2}{\zeta(2)}+O(T^{1-\varepsilon}).
	\end{aligned}
	$$
    \end{cor}
	    In this paper, $\varepsilon$ is an arbitrarily small positive constant which is not necessarily
		the same at each occurrence. 
	
	\section{Preliminaries}
	\subsection{Approximate functional equations} 
	
	Let $L(s,f)$ be an $L$-function of degree $d$ in the sense of \cite{iwaniec2021analytic}. More precisely, we have
	$$
	L(s, f)=\sum_{n \geq 1} \frac{\lambda_f(n)}{n^s}=\prod_p \prod_{j=1}^d\left(1-\frac{\alpha_j(p)}{p^s}\right)^{-1}
	$$
	with $\lambda_f(1)=1, \lambda_f(n) \in \mathbb{C}, \alpha_j(p) \in \mathbb{C}$, such that the series and Euler products are absolutely convergent for $\operatorname{Re}(s)>1$. There exist an integer $q(f) \geq 1$ and a gamma factor
	
	$$
	\gamma(s, f)=\pi^{-d s / 2} \prod_{j=1}^d \Gamma\left(\frac{s-\kappa_j}{2}\right)
	$$
	with $\kappa_j \in \mathbb{C}$ and $\operatorname{Re}\left(\kappa_j\right)<1 / 2$ such that the complete $L$-function
	$$
	\Lambda(s, f)=q(f)^{s / 2} \gamma(s, f) L(s, f)
	$$
	admits an analytic continuation to a meromorphic function for $s \in \mathbb{C}$ of order 1 with poles at most at $s=0$ and $s=1$. Moreover we assume that $L(s, f)$ satisfies the functional equation
	$$
	\Lambda(s, f)=\varepsilon(f) \Lambda(1-s, \bar{f}).
	$$
	The following lemma can be derived from \cite[Theorem 5.3]{iwaniec2021analytic}. Below we simplify the above notation by not displaying the dependence on $f$ and we write $q=q(f)$. 
	\begin{lem} With notation as above, we have
		\begin{equation}\label{app}
			\begin{aligned}
				L(1/2+i t,f)=\sum_{n\geq1}\frac{\lambda_{f}(n)}{n^{1/2+it}}W_{1/2+it}\left(\frac{n}{\sqrt{q}}\right)+\varepsilon(f){q}^{1/2-s}\frac{\gamma(1/2-it,f)}{\gamma(1/2+it,f)}\sum_{n\geq1}\frac{\overline{\lambda_{f}(n)}}{n^{1/2-it}}W_{1/2-it}\left(\frac{n}{\sqrt{q}}\right),
			\end{aligned}
		\end{equation}
		where for $x, c>0$, $W_{s}(x)$ is a smooth function defined by 
		$$
		\begin{aligned}
			W_{s}(x)=\frac{1}{2 \pi i}\int_{(c)}x^{-u}G(u)\frac{\gamma(s+u,f)}{\gamma(s,f)}\frac{\mathrm{~d} u}{u}
		\end{aligned}
		$$
		and $G(u)$ be any function which is holomorphic and bounded in the strip $-4<\operatorname{Re}(u)<4$, even, and normalized by $G(0)=1$.
	\end{lem}
    \begin{rem}
    (1). We have that $W_{s}(n)\ll\left(\frac{n}{t^{d/2}}\right)^{-A}$ for any $A>0$, thus the sums in \eqref{app} are essentially supported on $n<T^{d/2+\varepsilon}$. \\
    (2). Below, we use the notation $\sum_{N\,dyadic}$ to indicate a sum over $N=2^{k}$ for integers $k$. Let $V_{1}$ be a smooth function supported on $[4/5,11/5]$, satisfying the partition of unity condition
    $$
    \begin{aligned}
    \sum_{N\,dyadic} V_{1}\left(\frac{n}{N}\right)=1
    \end{aligned}
    $$
    for all $n\geq 1$.
    \end{rem}
	\begin{lem}\label{2} For $|t|\in[T,2T]$, we have
		\begin{equation}\label{huangapp}
			\begin{aligned}
				L(1/2+i t, f)&=\frac{1}{2 \pi i} \int_{\varepsilon-i T^{\varepsilon}}^{\varepsilon+i T^{\varepsilon}} \sum_{\substack{N \leq T^{d/2+\varepsilon} \\
						N\,d y a d i c}} \sum_{n \geq 1} \frac{\lambda_f(n)}{n^{1 / 2+i t+w}} V_1\left(\frac{n}{N}\right)\left(\frac{ t}{2\pi}\right)^{dw/2} e^{i\operatorname{sgn}(t) \pi dw / 4}q^{w/2} \frac{G(w)}{w} \mathrm{~d} w \\
				& +\frac{\varepsilon(f)}{2 \pi i} \int_{\varepsilon-i T^\varepsilon}^{\varepsilon+i T^\varepsilon} \sum_{\substack{N \leq T^{d/2+\varepsilon} \\
						N\,d y a d i c}} \sum_{n \geq 1} \frac{\overline{\lambda_f(n)}}{n^{1 / 2-i t+w}} V_1\left(\frac{n}{N}\right)\left(\frac{t}{2 \pi e}\right)^{-dit}\left(\frac{t}{2 \pi}\right)^{dw/2} \\
				& \times e^{i\pi\operatorname{sgn}(t)(-  dw / 4+ d/4+\sum_{j}\kappa_{j}/2)}q^{w/2-it} \frac{G(w)}{w} \mathrm{~d} w+O\left(T^{d/4-1+\varepsilon}\right).\\
			\end{aligned}
		\end{equation}
	\end{lem}
	\begin{proof} The calculation is the same as Huang \cite[Lemma 11]{huang2021rankin}, but the results are slightly different. As $|t| \rightarrow \infty$, Stirling's formula gives
		$$
		\Gamma(\sigma+i t)=\sqrt{2 \pi}|t|^{\sigma-\frac{1}{2}+i t} \exp \left(-\frac{\pi}{2}|t|-i t+i \operatorname{sgn}(t) \frac{\pi}{2}\left(\sigma-\frac{1}{2}\right)\right)\left(1+O\left(|t|^{-1}\right)\right) \text {. }
		$$
		
		\noindent Hence for $|t| \in[T, 2 T]$ with $T$ large, and $\varepsilon \leq \operatorname{Re}(w) \ll 1,|\operatorname{Im}(w)| \ll T^{\varepsilon}$, and $\kappa_j \ll 1$, we have
		\begin{equation}\label{stir1}
			\frac{\Gamma\left(\frac{1 / 2+i t+w-\kappa_j}{2}\right)}{\Gamma\left(\frac{1 / 2+i t-\kappa_j}{2}\right)}=\left(\frac{|t|}{2}\right)^{w / 2} e^{i\operatorname{sgn}(t)\pi w / 4}\left(1+O\left(T^{ \varepsilon-1}\right)\right),
		\end{equation}
		\begin{equation}\label{stir2}
			\frac{\Gamma\left(\frac{1 / 2-i t-\kappa_j}{2}\right)}{\Gamma\left(\frac{1 / 2+i t-\kappa_j}{2}\right)}=\left(\frac{|t|}{2 e}\right)^{-i t}e^{i\pi \operatorname{sgn}(t)(1/4+\kappa_{j}/2
				)}\left(1+O\left(T^{-1}\right)\right) .
		\end{equation}
		
		\noindent For $|t| \in[T, 2 T]$, by the approximate functional equation \eqref{app} we have
		$$
		\begin{aligned}
			L(1 / 2+i t, f)= & \frac{1}{2 \pi i} \sum_{n \geq 1} \frac{\lambda_f(n)}{n^{1 / 2+i t}}  \int_{(2)} \frac{\gamma(1 / 2+i t+w, f)}{\gamma(1 / 2+i t, f)} \frac{q^{w/2}}{n^w} \frac{G(w)}{w} \mathrm{~d} w \\
			+&\frac{\varepsilon(f)q^{-it}}{2 \pi i} \sum_{n \geq 1} \frac{\overline{\lambda_f(n)}}{n^{1 / 2-i t}}  \int_{(2)} \frac{\gamma( 1 / 2-i t+w,f)}{\gamma( 1 / 2+i t,f)} \frac{q^{w/2}}{n^w} \frac{G(w)}{w} \mathrm{~d} w ,
		\end{aligned}
		$$
		
		\noindent where $G(w)=e^{w^2}$. Shifting the lines of integration to $\operatorname{Re}(w)=\varepsilon$, no poles are encountered in this shifting. By \cite[Proposition 5.4]{iwaniec2021analytic} we have
		$$
		\begin{aligned}
			\sum_{n>T^{d/2+\varepsilon}}\frac{\lambda_{f}(n)}{n^{1/2+it}}W_{1/2+it}(n)\ll\sum_{n>T^{d/2+\varepsilon}}\left|\frac{\lambda_{f}(n)}{n^{1/2+it  }}\right|\left(\frac{n}{t^{d/2}}\right)^{-A}\ll T^{-A}.
		\end{aligned}
		$$
		Hence  we can truncate the $n$-sum at $n \leq T^{d/2+\varepsilon}$ for the first sum and at $n \leq T^{d/2+\varepsilon}$ for the second sum above with a negligible error. By a simple calculation we have
		$$
		\begin{aligned}
			\sum_{n\leq T^{d/2+\varepsilon}}\frac{\lambda_{f}(n)}{n^{1/2+it+\varepsilon}}\int_{\varepsilon+i T^{\varepsilon}}^{\varepsilon+i\infty}\frac{\gamma(1 / 2+i t+w, f)}{\gamma(1 / 2+i t, f)} \frac{1}{n^w} \frac{G(w)}{w} \mathrm{~d} w\ll \int_{ T^{\varepsilon}}^{\infty}e^{-u^{2}}\mathrm{~d}u\ll e^{-T^{2\varepsilon}}.
		\end{aligned}
		$$
		Hence we can truncate at $|\operatorname{Im}(w)| \leq T^{\varepsilon}$ with a negligible error term. By Stirling's formula \eqref{stir1} and \eqref{stir2}, and then by a smooth partition of unity, we proved \eqref{huangapp}.
	\end{proof}
    
    In this paper we give the proof of Theorem \ref{t1} for even Hecke-Maass forms, the details for odd forms being entirely similar. Thus throughout, $f$ will denote an even Hecke-Maass cusp form for the full modular group with Laplacian eigenvalue $\frac{1}{4}+\mu^{2}$, where $\mu$ is real. Let $\lambda_{f}(n)$ denote the eigenvalue of the $n$-th Hecke operator corresponding to $f$. For the even Hecke-Maass cusp form, we define the $L$-function
    \begin{equation*}
    	L(s,f)=\sum_{n \geq 1} \lambda_f(n)n^{-s}.
    \end{equation*}
    The above function means that $d=2$, $\lambda_{f}(n)=\overline{\lambda_{f}(n)}$ and $q=\varepsilon(f)=1$
   
   Let the $L$-function in \eqref{huangapp} be $\zeta^{2}(s)$, and for $t\in[T,2T]$ we have that
    \begin{equation}\label{zetaapp1}
    	\begin{aligned}
    	\begin{aligned}
    		\zeta^2(1 / 2-i t) & =\frac{1}{2 \pi i} \int_{\varepsilon-i T^{\varepsilon}}^{\varepsilon+i T^{\varepsilon}} \sum_{\substack{M \leqslant T^{1 /+\varepsilon} \\
    				M \text { dyadic }}} \sum_{m \geqslant 1} \frac{d(m)}{m^{1 / 2-i t+w}} V_1\left(\frac{m}{M}\right)\left(\frac{t}{2 \pi}\right)^w e^{-i \pi w / 2} \frac{G(w)}{w} \mathrm{~d} w \\
    		& +\frac{1}{2 \pi i} \int_{\varepsilon-i T^{\varepsilon}}^{\varepsilon+i T^{\varepsilon}} \sum_{\substack{M \leqslant T^{1+\varepsilon} \\
    				M \text { dyadic }}} \sum_{m \geqslant 1} \frac{d(m)}{m^{1 / 2+i t+w}} V_1\left(\frac{m}{M}\right)\left(\frac{t}{2 \pi e}\right)^{2 i t}\left(\frac{t}{2 \pi}\right)^w \\
    		& \times e^{i \pi w / 2-i \pi / 2} \frac{G(w)}{w} \mathrm{~d} w+O\left(T^{-1 / 2+\varepsilon}\right).
    	\end{aligned}
    	\end{aligned}
    \end{equation}
    
	\subsection{Sums of Fourier coefficients}
	The Ramanujan conjecture for the Fourier coefficients of $f$ is known on average. By the Rankin-Selberg theory, we have
	\begin{lem}\label{3} We have 
		$$
		\begin{aligned}
			\sum_{n\leq x}|\lambda_{f}(n)|^{2}\ll x .
		\end{aligned}
		$$
	\end{lem}

   The following result (see \cite[Theorem 8.1]{iwaniec2021spectral}) shows that Fourier coefficients are orthogonal to additive characters on average.
	\begin{lem}\label{5} For any real number $\alpha$ and $\varepsilon>0$, we have
		$$
		\begin{aligned}
			\sum_{n\leq x}\lambda_{f}(n)e(\alpha n)\ll x^{1/2+\varepsilon}.
		\end{aligned}
		$$
	\end{lem}
	For $F(x) \in \mathcal{C}_C(0, \infty)$, we define the integral transforms
	$$
	\begin{aligned}
		& \Phi_F^{+}(x)=\frac{-\pi}{\sin (\pi i \mu)} \int_0^{\infty} F(y)\left(J_{2 i \mu}(4 \pi \sqrt{x y})-J_{-2 i \mu}(4 \pi \sqrt{x y})\right) \mathrm{d} y, \\
		& \Phi_F^{-}(x)=4\cosh (\pi \mu) \int_0^{\infty} F(y) K_{2 i \mu}(4 \pi \sqrt{x y}) \mathrm{d} y.
	\end{aligned}
	$$
	We have the following Voronoi summation formula (see \cite[Theorem A.4]{kowalski2002rankin}).
		\begin{lem}\label{Vor} Let $q\in \mathbb{N}$ and $a\in \mathbb{Z}$ be such that $(a,q)=1$. For $N>0$, we have
		$$
		\begin{aligned}
		\sum_{n=1}^{\infty} \lambda_f(n) e\left(\frac{a n}{q}\right) F\left(\frac{n}{N}\right)=\frac{N}{q} \sum_{ \pm} \sum_{n=1}^{\infty} \lambda_f(n) e\left(\mp \frac{\bar{a} n}{q}\right) \Phi_F^{ \pm}\left(\frac{n N}{q^2}\right),
		\end{aligned}
		$$
		where $a \bar{a} \equiv 1(\bmod q)$.
	\end{lem}
	The function $\Phi_F^{ \pm}(x)$ has the following properties (see \cite[Lemma 3.4]{lin2021analytic}).
	 \begin{lem}\label{Phi} For any fixed integer $J\geq 1$ and $x\gg1$, we have $\Phi_F^{-}(x)\ll x^{-2024}$, and
	 	$$
	 	\begin{aligned}
	 		\Phi_F^{+}(x)=x^{-1 / 4} \int_0^{\infty} F(y) y^{-1 / 4} \sum_{j=0}^J \frac{c_j e(2 \sqrt{x y})+d_j e(-2 \sqrt{x y})}{(x y)^{j / 2}} \mathrm{~d} y+O_{\mu, J}\left(x^{-J / 2-3 / 4}\right),
	 	\end{aligned}
	 	$$
	 	where $c_{j}$ and $d_{j}$ are some constants depending on $\mu$.
	 \end{lem}
	 \begin{rem}
	 	For $x\ll1$, usually the bound $x^{j}J_{\nu}^{(j)}\ll_{\nu, j}1$ would be sufficient in applications.
	 \end{rem}
	\subsection{Estimation of integrals}
	We will use the following stationary phase lemma several times.
	\begin{lem}\label{6} Suppose $w$ is a smooth function with compact support on $[Z, 2 Z]$ such that $w^{(j)}(\xi) \ll(Z / X)^{-j}$. Suppose that on the support of $w$, $h$ is smooth and satisfies that $h^{(j)}(\xi) \ll \frac{Y}{Z^{j}}$, for all $j \geq 0$. Let
		
		$$
		I=\int_{\mathbb{R}} w(\xi) e^{i h(\xi)} \mathrm{d} \xi.
		$$
		(i) If $h^{\prime}(\xi) \gg \frac{Y}{Z}$ for all $\xi \in \operatorname{supp} w$. Suppose $Y / X \geq 1$. Then $I \ll_A Z(Y / X)^{-A}$ for $A$ arbitrarily large.\\
		(ii) If $h^{\prime \prime}(\xi) \gg \frac{Y}{Z^2}$ for all $\xi \in \operatorname{supp} w$, and there exists $\xi_0 \in \mathbb{R}$ such that $h^{\prime}\left(\xi_0\right)=0$. Suppose that $Y / X^2 \geq R\geq1$. Then we have
		
		$$
		I=\frac{e^{i h\left(\xi_0\right)}}{\sqrt{h^{\prime \prime}\left(\xi_0\right)}} W(\xi_0)+O_A(ZR^{-A}),\quad \text { for any } A>0,
		$$
		for some smooth functions $W(\xi)$ (depending on $A$) supported on $\xi\asymp Z$.
	\end{lem}
\begin{proof}
	See \cite[Lemma 3.1]{kiral2019oscillatory}. In our case, the functions $w$ and $W$ satisfy the $X$-inert condition as stated in \cite[Lemma 3.1]{kiral2019oscillatory}.
\end{proof}
	\begin{rem}\label{rem}
	To compute the main term, we will need to use a more explicit version of stationary phase method (see \cite[Proposition 8.2]{blomer2013distribution}). In particular, the function $W$ has an asymptotic expansion of the form
	$$
		\begin{aligned}
			W(\xi_{0})=\sum_{n \leq \varepsilon^{-1}A} p_n\left(\xi_0\right),\quad
			p_n\left(\xi_0\right)=\frac{\sqrt{2 \pi} e^{\pi i / 4}}{n!}\left(\frac{i}{2 h^{\prime \prime}\left(\xi_0\right)}\right)^n G^{(2 n)}\left(\xi_0\right),
		\end{aligned}
		$$
		where $A>0$ is arbitrary, and
		$$
		\begin{aligned}
		G(\xi)=w(\xi) e^{i H(\xi)}, \quad H(\xi)=h(\xi)-h\left(\xi_0\right)-\frac{1}{2} h^{\prime \prime}\left(\xi_0\right)\left(\xi-\xi_0\right)^2 .
	\end{aligned}
	$$
	The leading term comes from $n=0$.
	\end{rem}

	\begin{lem}\label{7} For any $T>0$ and any sequence of complex numbers $a_n$, we have 
		$$
		\begin{aligned}
			\int_0^T{\left|\sum_{n \leq N} a_n n^{i t}\right|} ^2 \mathrm{~d} t \ll(T+N) \sum_{n \leq N}\left|a_n\right|^2.
		\end{aligned}
		$$
	\end{lem}
	\begin{proof}
	See \cite[Theorem 9.1]{iwaniec2021analytic}. 
	\end{proof}
	
	\subsection{Averages of Kloosterman sums}
	Here and throughout, for $a,b,c\in \mathbb{Z}$, we write
	$$
	\begin{aligned}
	S(a, b ; c):=\sideset{}{^{*}}\sum_{x(\bmod c)}\mathrm{e}\left(\frac{a x+b \bar{x}}{c}\right)
	\end{aligned}
    $$
    for the classical Kloosterman sum.
    \begin{lem}\label{8}
    	Let $n, k\in \mathbb{Z}$ with $r\ge 1$, and let $W \in \mathcal{C}_c^{\infty}(0,\infty)$, Then
    	$$
    \begin{aligned}
    	\sum_{\substack{k\\(k,r)=1}} \mathrm{e}\left(\frac{n\overline{k }}{r}\right) W\left(\frac{k}{K}\right)=\frac{K}{r}  \sum_k S\left(k, n  ; r\right) \widehat{W}\left(\frac{k K}{r}\right) .
    \end{aligned}
    $$
    \end{lem}
    \begin{proof}
    	The result follows from the Lemma 7.2 in \cite{chandee2024sixth} by taking $\alpha=e\nu_1 g=1$.
    \end{proof}
    To deal with the averages of Kloosterman sums we shall use the following result \cite[Theorem 9]{deshouillers1982kloosterman}.
    \begin{lem}\label{9}
    	Let $C, M, N$ be positive real numbers and $g:{\mathbb{R}}^{3}\rightarrow\mathbb{R}$ be a smooth function with compact support on $[C,3C]\times[M,3M]\times[N,3N]$ such that
    
    $$
    \left|\frac{\partial^{j+k+l}}{\partial m^j \partial n^k \partial c^l} g(m, n, c)\right| \leqq M^{-j} N^{-k} C^{-t} \quad \text { for } 0 \leqq j, k, l \leqq 2 .
    $$
    For any $\varepsilon>0$ and complex sequences $\mathbf{a}, \mathbf{b}$ one has
    
    $$
    \begin{aligned}
    	\sum_{c}  \sum_m a_m \sum_n b_n g(m, n, c) S(m , \pm n ;  c)\ll C^{1+\varepsilon} \sqrt{M N}\left\|\mathbf{a}\right\|_2\left\|\mathbf{b}\right\|_2
    \end{aligned}
    $$
    where $\|\mathbf{a}\|_2=\sqrt{\sum_m\left|a_m\right|^2}$ and $\|\mathbf{b}\|_2=\sqrt{\sum_n\left|b_n\right|^2}$.
    \end{lem}
	\section{Mellin transform and Proof of Theorem \ref{t1}}
	
	Applying the approximate functional equations \eqref{huangapp} and \eqref{zetaapp1}, we have that the left hand side of \eqref{eqt1} equals 
	\begin{align}\label{eight}
	&\frac{1}{(2 \pi i)^2}\int_{\mathbb{R}}V\left(\frac{t}{T}\right)\Bigg(\int_{\varepsilon-i T^{\varepsilon}}^{\varepsilon+i T^{\varepsilon}} \sum_{\substack{N \leq T^{1+\varepsilon}\nonumber \\
			N\,d y a d i c}} \sum_{n \geq 1} \frac{\lambda_f(n)}{n^{1 / 2+i t+w}} V_1\left(\frac{n}{N}\right)\left(\frac{ t}{2\pi}\right)^{w} e^{i\pi w / 2}\frac{G(w)}{w} \mathrm{~d} w\\
	&+\int_{\varepsilon-i T^\varepsilon}^{\varepsilon+i T^\varepsilon} \sum_{\substack{N \leq T^{1+\varepsilon} \nonumber\\
			N\,d y a d i c}} \sum_{n \geq 1} \frac{\lambda_f(n)}{n^{1 / 2-i t+w}} V_1\left(\frac{n}{N}\right)\left(\frac{t}{2 \pi e}\right)^{-2it}\left(\frac{t}{2 \pi}\right)^{w}e^{i\pi(-  w / 2+ 1/2)} \frac{G(w)}{w} \mathrm{~d} w\nonumber\\
	&+O\left(T^{-1/ 2+\varepsilon}\right) \Bigg)\nonumber\\
	&\times\Bigg(\int_{\varepsilon-i T^{\varepsilon}}^{\varepsilon+i T^{\varepsilon}} \sum_{\substack{M \leq T^{1+\varepsilon} \\
			M\,d y a d i c}} \sum_{m \geq 1} \frac{d(m)}{m^{1 / 2-i t+w}} V_1\left(\frac{m}{M}\right)\left(\frac{ t}{2\pi}\right)^{ w} e^{- i \pi  w / 2} \frac{G(w)}{w} \mathrm{~d} w +O\left(T^{-1/ 2+\varepsilon}\right)\nonumber\\
	& +\int_{\varepsilon-i T^\varepsilon}^{\varepsilon+i T^\varepsilon} \sum_{\substack{M \leq T^{1 +\varepsilon} \\
			M\,d y a d i c}} \sum_{m \geq 1} \frac{d(m)}{m^{1 / 2+i t+w}} V_1\left(\frac{m}{M}\right)\left(\frac{t}{2 \pi e}\right)^{2i t}\left(\frac{t}{2 \pi}\right)^{ w}
	e^{i \pi( w / 2-1/2)}  \frac{G(w)}{w} \mathrm{~d} w\Bigg)\mathrm{~d} t.
	\end{align}
	For convenience, we will write \eqref{eight} as
	$$
	\begin{aligned}
	\frac{1}{(2 \pi i)^2}\int_{\mathbb{R}}V\left(\frac{t}{T}\right)\left[\left(S_1+S_2\right)\left(S_3+S_4\right)+S_5\right]\mathrm{~d} t,
	\end{aligned}
    $$
    where $S_5$ comes from the error term of \eqref{eight}. Multiplying out the summand above leads to several cross terms, which we will analyze one by one.
	\subsection{Mellin transform}
	\begin{lem}\label{Mellin}
		Suppose $F(x)\in \mathcal{C}_{c}^{\infty}(0, \infty)$ satisfies $F^{(j)}(x) \ll T^{j\varepsilon}$.  And let $\widetilde{\Phi}_{F}^{ \pm}(s)$ be a Mellin transform of $\Phi_{F}^{\pm}(x)$ as in the Voronoi summation formula. Then for $\operatorname{Re}(s)=\varepsilon$, we have
		$$
	\begin{aligned}
		\widetilde{\Phi}_{F}^{ \pm}(s)\ll \frac{T^{A\varepsilon}}{(1+|s|)^{A}}
	\end{aligned}
	$$
	for any $A\geq 1$.
	\end{lem}
	\begin{proof}
The proof for $\widetilde{\Phi}_{F}^{ -}(s)$ proceeds in the same way. We thus give details only for the case $\widetilde{\Phi}_{F}^{ +}(s)$. By repeated integration by parts, we have
$$
\begin{aligned}
	\widetilde{\Phi}_{F}^{ \pm}(s)=\prod_{i=1}^{j}(s+i-1)^{-1}\int_{0}^{\infty}x^{s+j-1}(\Phi_{F}^{+})^{(j)}(x)\mathrm{~d} x.
\end{aligned}
$$
We split the integral into three ranges:
\[
\int_{0}^{\infty}=\int_{0}^{1}+\int_{1}^{T^{\varepsilon}}+\int_{T^{\varepsilon}}^{\infty},
\]
and then estimate each of them separately. 

For $x\leq 1$. Using the small-$x$ bounds for Bessel functions of purely imaginary order, we have $J_{\nu}^{(j)}(x)\ll x^{-j}$. Applying the chain rule, this implies that $\frac{d^j}{dx^j} J_\nu(4\pi\sqrt{xy})
\ll_j x^{-j}$, uniformly for $y\asymp 1$. Consequently,
$$
\begin{aligned}
	\int_{0}^{1}x^{s+j-1}(\Phi_{F}^{+})^{(j)}(x)\mathrm{~d} x\ll\int_{0}^{1}x^{\varepsilon-1}\mathrm{~d} x=\varepsilon^{-1}.
\end{aligned}
$$

For $x\ge T^{\varepsilon}$.
In this range, we use the asymptotic expansion
\begin{equation}\label{jifen}
\begin{aligned}
	\Phi_{F}^{+}(x)=x^{-1 / 4} \int_0^{\infty} F(y) y^{-1 / 4} \sum_{i=0}^J \frac{c_i e(2 \sqrt{x y})+d_i e(-2 \sqrt{x y})}{(x y)^{i / 2}} \mathrm{~d} y+O(x^{-J / 2-3 / 4}).
\end{aligned}
\end{equation}
Differentiating with respect to $x$, we see that $(\Phi_{F}^{+})^{(j)}(x)$ can be written as a finite linear combination of oscillatory integrals with phase $\pm 2\sqrt{xy}$.
Applying the Lemma \ref{6} (i) (or repeated integration by parts), we obtain that for any $A\geq1$,
\[
(\Phi_{F}^{+})^{(j)}(x)\ll x^{-A},
\quad x\ge T^{\varepsilon}.
\]
Consequently,
\[
\int_{T^{\varepsilon}}^{\infty}
x^{s+j-1}(\Phi_{F}^{+})^{(j)}(x)\,dx
\ll T^{-A},
\]
and this contribution is negligible.

\medskip
For $1<x\le T^{\varepsilon}$.
From \eqref{jifen} and the chain rule with respect to $x$, we deduce that $\frac{d^j}{dx^j} e(2\sqrt{xy})\ll_j x^{-j/2}$, And then
\[
\int_{1}^{T^{\varepsilon}} x^{s+j-1}(\Phi_{F}^{+})^{(j)}(x)\,dx
\ll T^{j\varepsilon/2}.
\]
Combining the above three ranges, we conclude that
$\widetilde{\Phi}_{F}^{+}(s)$
satisfies the desired bound.
	
	\end{proof}
	\subsection{Proof of Theorem \ref{t1}}
	
	\begin{lem} Given any $\varepsilon>0$, we have
        $$
		\begin{aligned}
		&I_1\coloneq\frac{1}{(2 \pi i)^2}\int_{\mathbb{R}}V\left(\frac{t}{T}\right)S_{1}S_{3}\mathrm{~d} t=cT\frac{L(1,f)^2}{\zeta(2
		)}+O\left(T^{1/2+\varepsilon}\right),\\
    	&I_4\coloneq\frac{1}{(2 \pi i)^2}\int_{\mathbb{R}}V\left(\frac{t}{T}\right)S_{2}S_{4}\mathrm{~d} t=cT\frac{L(1,f)^2}{\zeta(2
    		)}+O\left(T^{1/2+\varepsilon}\right),
    \end{aligned}
    $$
    where $c$ is as defined in Theorem \ref{t1}.
    \end{lem}
	\begin{proof} 
		Since $S_1S_3$ and $S_2S_4$ are symmetric, it suffices to prove the case of $S_1S_3$. We can exchange the order of integral and summation, and rewrite $I_1$ as
		$$
		\begin{aligned}
			\frac{1}{(2\pi i)^2}&\int_{\varepsilon-iT^\varepsilon}^{\varepsilon+iT^\varepsilon}\int_{\varepsilon-iT^\varepsilon}^{\varepsilon+iT^\varepsilon}\sum_{\substack{N \leq T^{1+\varepsilon} \\
					N\,d y a d i c}} \sum_{n \geq 1}  \frac{\lambda_{f}(n)}{n^{1/2+u}}V_1\left(\frac{n}{N}\right)\sum_{\substack{M \leq T^{1+\varepsilon} \\
					M\,d y a d i c}} \sum_{m\geq 1} \frac{d(m)}{m^{1/2+v}}V_1\left(\frac{m}{M}\right)\\&\times{\left(\frac{1}{2\pi}\right)}^{u+v}e^{\frac{i\pi(u-v)}{2}}\frac{G(u)G(v)}{uv}\int_{\mathbb{R}}V\left(\frac{t}{T }\right)t^{u+v}e^{it\log\frac{m}{n}}\mathrm{~d} t\mathrm{~d} u\mathrm{~d} v.
		\end{aligned}
		$$
		We now focus on the $t$-integral. It is evident that the main contribution to $I_1$ comes from the terms where $m=n$. To proceed, we make the change of variable $t=T\xi$, yielding
		\begin{equation}\label{main}
			\begin{aligned}
				\frac{T}{(2\pi i)^2}&\int_{\varepsilon-iT^\varepsilon}^{\varepsilon+iT^\varepsilon}\int_{\varepsilon-iT^\varepsilon}^{\varepsilon+iT^\varepsilon}\int_{\mathbb{R}}\sum_{\substack{N \leq T^{1+\varepsilon} \\
						N\,d y a d i c}} \sum_{n \geq 1}  \frac{\lambda_{f}(n)d(n)}{n^{1+u+v}}V_1\left(\frac{n}{N}\right)V(\xi)  {\left(\frac{T\xi}{2\pi}\right)}^{u+v}\\&\times e^{\frac{i\pi(u-v)}{2}}\frac{G(u)G(v)}{uv}\mathrm{~d} \xi\mathrm{~d} u\mathrm{~d} v.
			\end{aligned}
		\end{equation}
		We now remove the truncation on the imaginary parts of the integrals, incurring only a negligible error. The assumption $N \leq T^{1+\varepsilon}$, imposed at the beginning of the proof, can be dropped at the cost of an admissible error term $O(T^{-2025})$. Consequently, the $n$-sum inside the integral collapses to $\frac{L(1+u+v,f)^{2}}{\zeta(2(1+u+v))}$.
		We then shift the contours of integration to $\operatorname{Re}(u)=\operatorname{Re}(v)=-1/4+\varepsilon$. In doing so, we encounter a simple pole at $u=v=0$, whose residue contributes $cT\frac{L(1,f)^{2}}{\zeta(2)}$. The remaining integrals along the shifted contours can be bounded in a standard way by $T^{1/2+\varepsilon}$.
		
		For $m\neq n$ in $I_1$, we note that the following estimate holds for $I_1$:
		$$
		\begin{aligned}
		I_1 \ll T^{\varepsilon} \sup _{u \in\left[\varepsilon-i T^{\varepsilon}, \varepsilon+i T^{\varepsilon}\right]} \sup _{\substack{N \leqslant T^{1+\varepsilon} \\ M \leqslant T^{1+\varepsilon}}}\left|I_a\right|,
		\end{aligned}
		$$
where 
$$
\begin{aligned}
I_a=\int_{\mathbb{R}} V\left(\frac{t}{T}\right) \sum_{n \geqslant 1} \frac{\lambda_f(n)}{n^{1 / 2+i t+u}} V_1\left(\frac{n}{N}\right) \sum_{m \geqslant 1} \frac{d(m)}{m^{1 / 2-i t+u}} V_1\left(\frac{m}{M}\right) t^{2 u} \mathrm{~d} t .
\end{aligned}
$$
Now for $I_a$ let $t=T \xi$. Then merge $t^{2 u}$ into $V$ to get $V_2$, thus we have

$$
I_a=T^{1+2 \varepsilon} \sum_{n \geqslant 1} \frac{\lambda_f(n)}{n^{1 / 2+u}} V_1\left(\frac{n}{N}\right) \sum_{m \geqslant 1} \frac{d(m)}{m^{1 / 2+u}} V_1\left(\frac{m}{M}\right) \int_{\mathbb{R}} \frac{V_2(\xi)}{T^{2 \varepsilon}} e^{i T \xi \log \frac{m}{n}} \mathrm{~d} \xi .
$$
\textbf{Case 1:} $M \ll T^{1-\varepsilon}$. Suppose $h_1(\xi)=T \xi \log \frac{m}{n}$, so that

$$
\begin{aligned}
	h_{1}^{\prime}(\xi) & =T \log \left(\frac{m}{n}\right) \\
	& \gg T \min \left\{\frac{|n-m|}{m}, 1\right\} \\
	& \gg T^{\varepsilon}, \\
	h_{1}^{(j)}(\xi) & =0 \ll T^{\varepsilon} \quad(j \geqslant 2) .
\end{aligned}
$$
Hence the integral in $I_a$ satisfies condition (i) in Lemma \ref{6}, where $Z=1$, $Y=T^{\varepsilon}$ and $X=\Delta$. Thus we have

$$
\int_{\mathbb{R}} \frac{V_2(\xi)}{T^{2 \varepsilon}} e^{i \xi T \log \frac{n}{m}} \mathrm{~d} \xi \ll T^{-2025}.
$$
Furthermore we have $I_a=O(T^{-2025})$.\\
\textbf{Case 2:} $T^{1-\varepsilon}\ll M\ll T^{1+\varepsilon}$. For $N\ll T^{1-2\varepsilon}$, applying Lemma \ref{6} (i) with $Z=1$, $Y=T\log T$ and $X=\Delta$, we obtain $I_a\ll T^{-2025}$. It remains to consider the complementary range $N\gg T^{1-\varepsilon}$.

We now write $m=ab$, and apply a smooth partition of unity to the variables $a$ and $b$, to see that our sum is now at most $\log^2 T$ sums of the form 
$$
\begin{aligned}
T\sumthree_{a,b,n}\frac{\lambda_{f}(n)}{(abn)^{1/2+u}}V_1\left(\frac{a}{A}\right)V_1\left(\frac{b}{B}\right)V_1\left(\frac{ab}{M}\right)V_1\left(\frac{n}{N}\right)\int_{\mathbb{R}} V_{2}(\xi)e^{i T \xi \log \frac{ab}{n}} \mathrm{~d} \xi ,
\end{aligned}
$$
where $AB\asymp M$. Without loss of generality we may assume that $A\geq B$. Now we remove the weight function $V_1\left(\frac{ab}{M}\right)$ by a Mellin inversion, getting
$$
\begin{aligned}
	T\sumthree_{a,b,n}\frac{\lambda_{f}(n)}{(abn)^{1/2+u}}V_1\left(\frac{a}{A}\right)V_1\left(\frac{b}{B}\right)V_1\left(\frac{n}{N}\right)\int_{(\varepsilon)}\widetilde {V_{1}}(s)\left(\frac{ab}{M}\right)^{-s}\int_{\mathbb{R}} V_{2}(\xi)e^{i T \xi \log \frac{ab}{n}} \mathrm{~d} \xi\mathrm{~d} s.
\end{aligned}
$$
By integrating by parts several times, we see that it suffices to bound 
$$
\begin{aligned}
	T\sumthree_{a,b,n}\frac{\lambda_{f}(n)}{(ab)^{1/2+u+s}n^{1/2+u}}V_1\left(\frac{a}{A}\right)V_1\left(\frac{b}{B}\right)V_1\left(\frac{n}{N}\right)\int_{\mathbb{R}} V_{2}(\xi)e^{i T \xi \log \frac{ab}{n}} \mathrm{~d} \xi,
\end{aligned}
$$
for $|s|\ll T^{\varepsilon}$ and $\operatorname{Re}(s)=\varepsilon$. We will further allow ourselves to rewrite the above as
$$
\begin{aligned}
	T\sumthree_{a,b,n}\frac{\lambda_{f}(n)}{(abn)^{1/2}}V_1\left(\frac{a}{A}\right)V_1\left(\frac{b}{B}\right)V_1\left(\frac{n}{N}\right)\int_{\mathbb{R}} V_{2}(\xi)e^{i T \xi \log \frac{ab}{n}} \mathrm{~d} \xi,
\end{aligned}
$$
for slightly different functions $V_1$. We may absorb a factor of $\frac{\sqrt{ABN}}{\sqrt{abn}}$ into the smooth functions $V_1$ without affecting their essential properties. Consequently, it suffices to consider the expression
\begin{equation}
\begin{aligned}
	\frac{T}{\sqrt{ABN}}\sumthree_{a,b,n}\lambda_{f}(n)V_1\left(\frac{a}{A}\right)V_1\left(\frac{b}{B}\right)V_1\left(\frac{n}{N}\right)\int_{\mathbb{R}} V_{2}(\xi)e^{i T \xi \log \frac{ab}{n}} \mathrm{~d} \xi,
\end{aligned}
\end{equation}

We first deal with the sum over $a$. Making a change of variable and applying Poisson summation, we obtain that the $a$-sum is equal to
\begin{equation}
\begin{aligned}\label{a-sum}
	A\sum_{a}\int_{\mathbb{R}}V_{1}(x)e^{i (T \xi \log \frac{Abx}{n}-2\pi Aax)}\mathrm{~d} x.
\end{aligned}
\end{equation}
Suppose
$$
\begin{aligned}
h_2(x)=T \xi \log \frac{Abx}{n}-2\pi Aax.
\end{aligned}
$$
Then we have
$$
\begin{aligned}
	h_2^{\prime}(x) & = \frac{T\xi}{x}-2\pi Aa ,\\
	h_2^{\prime \prime}(x) & = -\frac{T\xi}{x^2}, \quad h_2^{(j)}(v) \asymp_j T \, (j \geq 2) .
\end{aligned}
$$
The solution of $h^{\prime}(x)=0$ is $x_{0}=\frac{T\xi}{2\pi Aa}$. Note that $h_{2}(x_0)=T\xi\log\frac{bT\xi}{2\pi ane}$. By the stationary phase method (see Lemma \ref{6} (ii)), we get
$$
\begin{aligned}
	\eqref{a-sum}=\frac{A}{\sqrt{T}}\sum_{a}e\left(\frac{T\xi}{2\pi}\log\frac{bT\xi}{2\pi ane}\right)W_{1}\left(\frac{T\xi}{Aa}\right)+O(T^{-2025}),
\end{aligned}
$$
where $W_1(x)$ is a smooth function supported on $x\asymp 1$ and $W_{1}^{(i)}(x)\ll 1$ for $i\geq 0$. Hence we only need to consider $a\asymp T/A$, otherwise the contribution is negligibly small. 
 
We apply a smooth partition of unity in $a$, introducing $V_{1}\left(\frac{a}{A_0}\right)$. This yields at most $O(\log T)$ sums of the form
$$
\begin{aligned}
\frac{A}{\sqrt{T}}\sum_{a}e\left(\frac{T\xi}{2\pi}\log\frac{bT\xi}{2\pi ane}\right)W_{1}\left(\frac{T\xi}{Aa}\right)V_{1}\left(\frac{a}{A_0}\right).
\end{aligned}
$$
The weight function $W_1$ implies that $A_0\asymp T/A$. We may again remove the factor $W_{1}\left(\frac{T\xi}{Aa}\right)$ by Mellin inversion, exactly as in the treatment of $V_{1}\left(\frac{ab}{M}\right)$. We now turn to the $\xi$-integral. Let
$$
\begin{aligned}
	h_3(\xi)=T \xi \log \frac{bT\xi}{2\pi ane}.
\end{aligned}
$$
Then we have
$$
\begin{aligned}
	h_3^{\prime}(\xi) & = T\log\frac{bT\xi}{2\pi an},\\
	h_3^{\prime \prime}(\xi) & = \frac{T}{\xi}, \quad h_2^{(j)}(v) \asymp_j T, \, (j \geq 2) .
\end{aligned}
$$
The solution of $h_{3}^{\prime}(\xi)=0$ is $\xi_{0}=\frac{2\pi an}{bT}$. Note that $h_{3}(\xi_0)=-\frac{2\pi an}{b}$. By Lemma \ref{6} (ii) for the $\xi$-integral, we obtain
$$
\xi-\text { integral }=\frac{e\left(-\frac{an}{b}\right)}{\sqrt{T}} W_2\left(\frac{an}{bT}\right)+O\left(T^{-2025}\right),
$$
where $W_2(\xi)$ is a smooth function supported on $[1, 2]$ and satisfies $W_{2}^{(i)}(\xi)\ll \Delta^i$ for $i\geq 0$. Applying Mellin inversion once more, we remove the weight function $W_2$. Writing $d=(a,b)$, we decompose $a=da'$ and $b=db'$ with $(a',b')=1$. For a fixed $d$, we denote by
$$
	\begin{aligned}
		\mathcal{S}(d)\coloneq\sumthree_{\substack{a,b,n\\(a,b)=1}}\lambda_{f}(n)e\left(-\frac{an}{b}\right)V_1\left(\frac{daA}{T}\right)V_1\left(\frac{db}{B}\right)V_1\left(\frac{n}{N}\right).
	\end{aligned}
$$
It now remains to examine
\begin{equation}\label{key2}
	\begin{aligned}
		{\sqrt{\frac{A}{BN}}}
		\sum_{d\ll T^{1+\varepsilon}} \mathcal{S}(d).
	\end{aligned}
\end{equation}

Applying Voronoi summation, we now transform the $n$-sum in \eqref{key2} into
\begin{equation}\label{key2-n}
	\begin{aligned}
		\frac{Nd}{B}\sum_{ \pm} \sum_{n=1}^{\infty} \lambda_f(n) e\left(\mp \frac{\overline{a}n }{b}\right) \Phi_{V_1}^{ \pm}\left(\frac{n N}{b^2}\right).
	\end{aligned}
\end{equation}
We split the analysis into two cases. We first assume that $nN/b^2\gg T^{\varepsilon}$. By Lemma \ref{Phi}, the contribution from $\Phi_{V_1}^{-}\left(\frac{n N}{b^2}\right)$ is negligible. Therefore, it remains to consider $\Phi_{V_1}^{+}(x)$, which is given by
$$
\begin{aligned}
	\Phi_{V_1}^{+}(x)=x^{-1 / 4} \int_0^{\infty} V_1(y) y^{-1 / 4} \sum_{j=0}^J \frac{c_j e(2 \sqrt{x y})+d_j e(-2 \sqrt{x y})}{(x y)^{j / 2}} \mathrm{~d} y+O(x^{-J / 2-3 / 4}).
\end{aligned}
$$
Note that $4\sqrt{xy}^{(i)}(y)\asymp \sqrt{x}$ for $y\asymp 1$. Hence by Lemma~\ref{6} (i), we have $\Phi_{V_1}^{+}\ll T^{-2025}$. It is also worth noting that $B\gg T^{1/2-\varepsilon}$. We now assume instead that $nN/b^2\ll T^{\varepsilon}$. Under this assumption, we may remove $\Phi_{V_1}^{ \pm}$ by Mellin inversion, and obtain
$$
\begin{aligned}
	\Phi_{V_1}^{ \pm}\left(\frac{nN}{b^2}\right)=\frac{1}{2\pi i}\int_{(\varepsilon)}\widetilde{\Phi}_{V_1}^{ \pm}(s)\left(\frac{nN}{b^2}\right)^{-s}\mathrm{~d} s.
\end{aligned}
$$
In view of Lemma  \ref{Mellin}, the estimate for \eqref{key2-n} is equivalent to that for 
$$
\begin{aligned}
	\frac{Nd}{B}\sum_{ \pm} \sum_{nN/b^2\ll T^{\varepsilon}} \frac{\lambda_f(n)}{n^s} b^{2s}e\left(\mp \frac{\overline{a}n }{b}\right) 
\end{aligned}
$$
for $|s|\ll T^{\varepsilon}$ and $\operatorname{Re}(s)=\varepsilon$. We now break the $n$-sum using a smooth partition of unity and insert the above expression into \eqref{key2}. Since $\operatorname{Im}(s)\ll T^{\varepsilon}$, the factors involving $s$ can be absorbed into the smooth weight functions. Consequently, we have
$$
\begin{aligned}
		\mathcal{S}(d)\ll \frac{T^{\varepsilon}Nd}{B}\sup_{N_0}\left|\sumthree_{a,b,n}\lambda_{f}(n)e\left(\mp\frac{\overline{a}n}{b}\right)V_1\left(\frac{daA}{T}\right)V_1\left(\frac{db}{B}\right)V_1\left(\frac{n}{N_0}\right)\right|,
\end{aligned}
$$
where $N_0\ll \frac{T^{\varepsilon} b^{2}}{N}\asymp \frac{T^{\varepsilon}B^2}{d^{2}N}\ll T^{\varepsilon}/d^2$.

We now proceed to bound the above sum. Lemma \ref{8} yields
$$
\begin{aligned}
	\sum_{a}e\left(\mp\frac{\overline{a}n}{b}\right)V_1\left(\frac{daA}{T}\right)=\frac{T}{Adb}\sum_{a}S(a, n; b)\widehat{V_1}\left(\frac{a T}{Adb}\right).
\end{aligned}
$$
By repeated integration by parts, we have $\widehat{V_1}(x)\ll T^{-2025}$ for $x\gg T^{\varepsilon}$. This allows us to truncate the $a$-sum at $a\ll AB/T^{1-\varepsilon}$. In what follows, we remove $\widehat{V_1}$ by Mellin inversion. We now write $\widetilde{\widehat{V_1}}(s)$ as the Mellin transform of $\widehat{V_1}$. Using the decay properties of $\widehat{V_1}$ and repeated integration by parts, we have
$$
\begin{aligned}
\widetilde{\widehat{V_1}}(s)\ll\frac{T^{\varepsilon}}{|s|^{j}}
\end{aligned}
$$
for $\operatorname{Re}(s)=\varepsilon$ and any $j\geq 1$. In fact, we can truncate $s$-integral at $|s|\ll T^{\varepsilon}$ with a negligible error term. Applying a smooth partition of unity, we arrive at
$$
	\begin{aligned}
		\mathcal{S}(d)\ll\frac{T^{1+\varepsilon}Nd}{AB^2}\left|\sumthree_{a,b,n}\lambda_{f}(n)S(a,n;b)V_{1}\left(\frac{a}{A_1}\right)V_1\left(\frac{db}{B}\right)V_1\left(\frac{n}{N_0}\right)\right|,
	\end{aligned}
$$
where $A_1\ll AB/T^{1-\varepsilon}\ll T^{\varepsilon}$. Finally applying the Lemma \ref{9} we have that
$$
\begin{aligned}
	\mathcal{S}(d)\ll\frac{1}{d^{1+\varepsilon}}\cdot\frac{NT^{1+\varepsilon}}{AB}.
\end{aligned}
$$
Inserting this estimate back into \eqref{key2} and summing over $d$, we deduce that $\eqref{key2}\ll T^{1/2+\varepsilon}$.

Having said all of above, this completes the proof.
	\end{proof}

	\begin{lem} Given any $\varepsilon>0$, we have
	$$
	\begin{aligned}
		&I_2\coloneq\frac{1}{(2 \pi i)^2}\int_{\mathbb{R}}V\left(\frac{t}{T}\right)S_{1}S_{4}\mathrm{~d} t=O\left(T^{1/2+\varepsilon}\right),\\
		&I_3\coloneq\frac{1}{(2 \pi i)^2}\int_{\mathbb{R}}V\left(\frac{t}{T}\right)S_{2}S_{3}\mathrm{~d} t=O\left(T^{1/2+\varepsilon}\right).
	\end{aligned}
	$$
	\end{lem}
	\begin{proof}
		We only give the proof of The $I_2$, and the same is true for $I_3$. We note that $I_2$ has the following estimate:
		
		$$
		I_2 \ll T^{\varepsilon} \sup _{u \in\left[\varepsilon-i T^{\varepsilon}, \varepsilon+i T^{\varepsilon}\right]} \sup _{\substack{N \leqslant T^{1+\varepsilon} \\ M \leqslant T^{1+\varepsilon}}}\left|I_b\right|,
		$$
		where
		
		$$
		I_b=\int_{\mathbb{R}} V\left(\frac{t}{T}\right) \sum_{n \geqslant 1} \frac{\lambda_f(n)}{n^{1 / 2+i t+u}} V_1\left(\frac{n}{N}\right) \sum_{m \geqslant 1} \frac{d(m)}{m^{1 / 2+i t+u}} V_1\left(\frac{m}{M}\right)\left(\frac{t}{2 \pi e}\right)^{2 i t} t^{2 u} \mathrm{~d} t
		$$
		Let $t=T \xi$. Then merge merge $t^{2 u}$ into $V$ to get $V_2$, thus we have
		
		$$
		I_b=T \sum_{n \geqslant 1} \frac{\lambda_f(n)}{n^{1 / 2+u}} V_1\left(\frac{n}{N}\right) \sum_{m \geqslant 1} \frac{d(m)}{m^{1 / 2+u}} V_1\left(\frac{m}{M}\right) \int_{\mathbb{R}} V_2(\xi) e^{i T \xi(2\log\frac{T\xi}{2\pi e}-\log mn)} \mathrm{d} \xi .
		$$
		Suppose 
		$$
		\begin{aligned}
		h_4(\xi)=2T\xi\log\frac{T\xi}{2\pi e\sqrt{mn}}.
		\end{aligned}
		$$
		Then we have
		$$
		\begin{gathered}
			h_{4}^{\prime}(\xi)=2 T \log\frac{T\xi}{2\pi\sqrt{mn}},\\
			h_{4}^{\prime \prime}(\xi)=2 \frac{T}{\xi}, \quad h_{4}^{(j)}(\xi)\asymp T\,(j \geqslant 2).
		\end{gathered}
		$$
		By repeated integration by parts we obtain $I_b\ll T^{-2025}$ unless $MN\asymp T^2$ which we assume now. The solution of $h_{4}^{\prime}(\xi_0)=0$ is $\xi_0=\frac{2\pi\sqrt{mn}}{T}$. Note that $h_4(\xi_0)=-4\pi\sqrt{mn}$. By Lemma \ref{6} (ii) for the $\xi$-integral, we obtain
		
		$$
		I_{b}=\sum_{m \geqslant 1} \frac{d(m)}{m^{1 / 4+u}} V_1\left(\frac{m}{M}\right) \sum_{n \geqslant 1} \frac{\lambda_f(n)}{n^{1 / 4+u}} V_1\left(\frac{n}{N}\right) e(-2\sqrt{mn}) W_2\left(\frac{2 \pi \sqrt{m n}}{T}\right),
		$$
		where $W_2(\xi)$ is a smooth function supported on $[1, 2]$ and $W_{2}^{(i)}(\xi)\ll \Delta^i$ for $i\geq 0$. We now write $m=ab$, and apply a smooth partition of unity to $a$ and $b$. As before, this shows that $I_b$ is bounded by $O(\log^2 T)$ sums of the form 
		$$
		\sumthree_{a,b,n} \frac{\lambda_f(n)}{(abn)^{1 / 4+u}} V_1\left(\frac{a}{A}\right) V_1\left(\frac{b}{B}\right)V_1\left(\frac{ab}{M}\right)V_1\left(\frac{n}{N}\right) e(-2\sqrt{abn}) W_2\left(\frac{2 \pi \sqrt{ab n}}{T}\right),
		$$
		where $AB\asymp M$. We may again neglect the factor $V_1\left(\frac{ab}{M}\right)$ and $W_2\left(\frac{2\pi\sqrt{abn}}{T}\right)$, and absorb a factor of $\frac{(ABN)^{1/4}}{(abn)^{1/4}}$ into the smooth functions $V_1$. Thus we will estimate
		$$
		\begin{aligned}
	T^{-1/2}\sumthree_{a,b,n} \lambda_{f}(n)e(-2\sqrt{abn})V_1\left(\frac{a}{A}\right) V_1\left(\frac{b}{B}\right)V_1\left(\frac{n}{N}\right).
		\end{aligned}
		$$
		
		Now we consider the sum over $a$. Making a change of variable and applying Poisson summation, we obtain that the $a$-sum is equal to
		\begin{equation}
			\begin{aligned}\label{a-sum2}
				A\sum_{a}\int_{\mathbb{R}}V_{1}(x)e(-2\sqrt{Abnx}-Aax)\mathrm{~d} x.
			\end{aligned}
		\end{equation}
		Suppose
		$$
		\begin{aligned}
			h_5(x)=-4\pi \sqrt{Abnx}-2\pi Aax.
		\end{aligned}
		$$
		Then we have
		$$
		\begin{aligned}
			h_5^{\prime}(x) & = -2\pi\sqrt{\frac{Abn}{x}}-2\pi Aa, \\
			h_5^{\prime \prime}(x) & = \pi\sqrt\frac{Abn}{x^3}, \quad h_5^{(j)}(v) \asymp_j T \, (j \geq 2) .
		\end{aligned}
		$$
		The solution of $h_{5}^{\prime}(x)=0$ is $x_{0}=\frac{bn}{Aa^2}$. Note that $h_{5}(x_0)=\frac{2\pi bn}{a}$. By the stationary phase method (see Lemma \ref{6} (ii)), we get
		$$
		\begin{aligned}
			\eqref{a-sum2}=\frac{A}{\sqrt{T}}\sum_{a}e\left(\frac{bn}{a}\right)W_{1}\left(\frac{bn}{Aa^2}\right)+O(T^{-2025}),
		\end{aligned}
		$$
		where $W_1(x)$ is a smooth function supported on $x\asymp 1$ and $W_{1}^{(i)}(x)\ll 1$ for $i\geq 0$. Hence we only need to consider $Aa^2\asymp BN$, otherwise the contribution is negligibly small. Inserting this estimate back into the previous expression, we are led to consider
		\begin{equation}
			\frac{A}{T}\sumthree_{a,b,n} \lambda_{f}(n)e\left(\frac{bn}{a}\right)W_1\left(\frac{bn}{Aa^2}\right) V_1\left(\frac{b}{B}\right)V_1\left(\frac{n}{N}\right).
		\end{equation}
		
		The above sum can be treated in the same manner as \eqref{key2} in the previous lemma. Since no new difficulty arises, we omit the details and obtain the same bound.
	\end{proof}

	\begin{lem} Given any $\varepsilon>0$, we have
	$$
	\begin{aligned}
			R\coloneq\frac{1}{(2 \pi i)^2}\int_{\mathbb{R}}V\left(\frac{t}{T}\right)S_{5}\mathrm{~d} t=O(T^{1/2+\varepsilon}).
	\end{aligned}
	$$
	\end{lem}
	\begin{proof}
		We can convert $R$ into the following form:
		$$
		\begin{aligned}
			R\ll T^{\varepsilon } \sup _{u \in\left[\varepsilon-i T^{\varepsilon}, \varepsilon+i T^{\varepsilon}\right]} \sup _{\substack{N \leq T^{1+\varepsilon} \\ M \leq T^{1+\varepsilon}}}|R_{1}+R_{2}+R_{3}|,
		\end{aligned}
		$$
		where 
		$$
		\begin{aligned}
			&R_{1}=T^{-1/2}\int_{\mathbb{R}}V\left(\frac{t}{T}\right)\left|\sum_{n\geq1}\frac{\lambda_{f}(n)}{n^{1/2+it+u}}V_1\left(\frac{n}{N}\right)t^{u}\right|\mathrm{~d} t,\\
			&R_{2}=T^{-1/2}\int_{\mathbb{R}}V\left(\frac{t}{T}\right)\left|\sum_{m\geq1}\frac{d(m)}{m^{1/2-it+u}}V_1\left(\frac{m}{M}\right)t^{u}\right|\mathrm{~d} t,\\ 
			&R_{3}=T^{-1}\int_{\mathbb{R}}V\left(\frac{t}{T}\right)\mathrm{~d} t.
		\end{aligned}
		$$
		For $R_{1}$, by Cauchy-Schwarz inequality, Lemma \ref{3} and Lemma \ref{7}, we have
		$$ 
		\begin{aligned}
	\begin{aligned}
		R_{1} & \ll T^{-1/2+\varepsilon}\left(\int_T^{2 T} 1^2 \mathrm{~d} t\right)^{1/2}\left(\int_T^{2 T}\left|\sum_{n \geq 1} V_1\left(\frac{n}{N}\right) \frac{\lambda_{f}(n)}{n^{1/2+it+u}}\right|^2 \mathrm{~d} t\right)^{1/2} \\
		& \ll T^{-1/2+\varepsilon} \cdot T^{1/2} \cdot(T+N)^{1/2}\\
		& \ll T^{1/2+\varepsilon} .
	\end{aligned}
		\end{aligned}
		$$ 
		The rest of the proof is similar to the proof of $R_1$ (the largest contribution arising from $R_1$ and $R_2$), hence it suffices to estimate $R_1$. Finally we can obtain that the upper bound of $R$ is $O(T^{1/2+\varepsilon})$.  
	\end{proof}
	
	Theorem \ref{t1} follows upon combining the previous lemmas.
	
	\section{Proof of Corollaries}
	Note that from \eqref{Ivic}, we have
	$$
	\begin{aligned}
		\int_{T}^{T+T/\Delta}\left|\zeta\left(\frac{1}{2}-it\right)\right|^4\mathrm{~d} t\ll \frac{T(\log T)^4}{\Delta}+T^{2/3}(\log T)^8.
	\end{aligned}
    $$
    In particular, if $\Delta\leq T^{1/3}(\log T)^{-4}$, we have
    $$
    \begin{aligned}
    	\int_{T}^{T+T/\Delta}\left|\zeta\left(\frac{1}{2}-it\right)\right|^4\mathrm{~d} t\ll \frac{T(\log T)^4}{\Delta};
    \end{aligned}
    $$
    see also \cite{iwaniec1980fourier}. Similarly, under the assumption $\Delta\leq (T\log T)^{1/3}$ and using \eqref{good}, we have
    $$
    \begin{aligned}
		\int_{T}^{T+T/\Delta}\left|L\left(\frac{1}{2}+it, f\right)\right|^2\mathrm{~d} t\ll\frac{T\log T}{\Delta}.
	\end{aligned}
	$$
	Similar to \cite[section 6]{lin2021analytic}, we choose the smooth function $V$ to be supported on $[1,2]$ and $V(x)=1$ on $[1+1/\Delta,2-1/\Delta]$. Then Theorem \ref{t1} yields
	$$
	\begin{aligned}
		\int_{T}^{2T}L\left(\frac{1}{2}+it, f\right)&{\zeta\left(\frac{1}{2}-it\right)}^2\mathrm{~d} t=2T \frac{L(1, f)^2}{\zeta(2)}+O\left(\frac{T}{\Delta}\right)+O\left(T^{1/2+\varepsilon}\right)\\&+\int_{T}^{T+T/\Delta}\left(1-V\left(\frac{t}{T}\right)\right)L\left(\frac{1}{2}+it, f\right){\zeta\left(\frac{1}{2}-it\right)}^2\mathrm{~d} t\\&+\int_{2T-T/\Delta}^{2T}\left(1-V\left(\frac{t}{T}\right)\right)L\left(\frac{1}{2}+it, f\right){\zeta\left(\frac{1}{2}-it\right)}^2\mathrm{~d} t.
	\end{aligned}
	$$
	From Cauchy-Schwarz inequality, this implies that
	$$
	\begin{aligned}
		\int_{T}^{2T}L\left(\frac{1}{2}+it, f\right)&{\zeta\left(\frac{1}{2}-it\right)}^2\mathrm{~d} t \\&=2 T \frac{L(1, f)^2}{\zeta(2)}+O\left(T^{1/2+\varepsilon}\right)+O\left(\frac{T(\log T)^{5/2}}{\Delta}\right) .
	\end{aligned}
	$$
	Corollary \ref{c3} then follows by choosing $\Delta=T^{\varepsilon}$.

	\section*{Acknowledgements}The author thanks Prof. Jianya Liu for his encouragement. She is especially grateful to Prof. Yongxiao Lin for proposing the project and for his guidance throughout the work. She also thanks Prof. Bingrong Huang for helpful comments. This work was supported by the National Key R\&D Program of China (No. 2021YFA1000700).
	
	\bibliographystyle{abbrv}
	%\bibliography{ref}

\begin{thebibliography}{10}
		
		\bibitem{blomer2013distribution}
		V.~Blomer, R.~Khan and M.~Young.
		\newblock Distribution of mass of holomorphic cusp forms.
		\newblock {\em Duke Mathematical Journal}, 162(14):2609--2644,
		2013.
		
		\bibitem{blomer2017moments}
		V.~Blomer, {\'E}.~Fouvry, E.~Kowalski, P.~Michel and D.~Mili{\'c}evi{\'c}.
		\newblock On moments of twisted $L$-functions.
		\newblock {\em American Journal of Mathematics}, 139(3):707--768,
		2017.
		
		\bibitem{bourgain2018decoupling}
		J.~Bourgain and N.~Watt.
		\newblock Decoupling for perturbed cones and the mean square of $|\zeta(\frac{1}{2}+it)|$.
		\newblock {\em International Mathematics Research Notices},
		2018(17):5219--5296, 2018.
		
		\bibitem{chandee2024sixth}
		V.~Chandee, X.~Li, K.~Matom\"aki and M.~Radziwi\l\l.
		\newblock The sixth moment of Dirichlet $L$-functions at the central point.
		\newblock arXiv:2409.01457 [math.NT], 
		2024.
		
		\bibitem{Chen2026Integral}
		Z.~Chen.
		\newblock Integral moment of the Riemann zeta function and Hecke $L$-functions.
		\newblock {\em International Journal of Number Theory},
		\newblock \url{https://doi.org/10.1142/S1793042126500211},
		2025
		
		\bibitem{das2015simultaneous}
		S.~Das and R.~Khan. 
		\newblock Simultaneous nonvanishing of Dirichlet $L$-functions and twists of Hecke-Maass $L$-functions.
		\newblock {\em J. Ramanujan Math. Soc},
		30(3):237--250, 2015.
		
		\bibitem{deshouillers1982kloosterman}
		J.~Deshouillers and H.~Iwaniec.
		\newblock Kloosterman sums and Fourier coefficients of cusp forms.
		\newblock {\em Inventiones mathematicae}, 70(2):219--288,
		1982.
		
		\bibitem{good1982square}
		A.~Good.
		\newblock The square mean of {D}irichlet series associated with cusp forms.
		\newblock {\em Mathematika}, 29(2):278--295, 1982.
		
		\bibitem{heath1979fourth}
		D.~R. Heath-Brown.
		\newblock The fourth power moment of the Riemann zeta function.
		\newblock {\em Proceedings of the London Mathematical Society}, 3(3):385--422, 1979.
		
		\bibitem{huang2021rankin}
		B.~Huang.
		\newblock On the {R}ankin--{S}elberg problem.
		\newblock {\em Mathematische Annalen}, 381(3):1217--1251, 2021.
		
		\bibitem{ingham1928mean}
		A.~E. Ingham.
		\newblock Mean-value theorems in the theory of the {R}iemann zeta-function.
		\newblock {\em Proceedings of the London Mathematical Society}, 2(1):273--300,
		1928.
		
		\bibitem{ivic1995fourth}
		A.~Ivi{\'c}.
		\newblock On the fourth moment of the {R}iemann zeta functions.
		\newblock {\em Publications de l'Institut Math{\'e}matique}, 57(77):101--110,
		1995.
		
		\bibitem{ivic1995fourth2}
		A.~Ivi{\'c} and Y.~Motohashi.
		\newblock On the fourth power moment of the {R}iemann zeta-function.
		\newblock {\em Journal of Number Theory}, 51(1):16--45, 1995.
		
		\bibitem{iwaniec1980fourier}
		H.~Iwaniec.
		\newblock {F}ourier coefficients of cusp forms and the {R}iemann zeta-function.
		\newblock {\em Seminaire de Th{\'e}orie des Nombres de Bordeaux}, 1--36, 1980.
		
		\bibitem{iwaniec2021spectral}
		H.~Iwaniec.
		\newblock Spectral methods of automorphic forms, volume~53.
		\newblock {\em American Mathematical Society}, Revista Matem{\'{a}}tica Iberoamericana (RMI), Madrid, Spain, 2002.
		
		\bibitem{iwaniec2021analytic}
		H.~Iwaniec and E.~Kowalski.
		\newblock Analytic number theory, volume~53.
		\newblock {\em American Mathematical Society}, 2004.
		
		\bibitem{keating2000random}
		J.~P. Keating and N.~C. Snaith.
		\newblock Random matrix theory and $\zeta(1/2+ it)$.
		\newblock {\em Communications in Mathematical Physics}, 214(1):57--89, 2000.
		
		\bibitem{kiral2019oscillatory}
		E.~M. K{\i}ral, I.~Petrow and M.~P. Young.
		\newblock Oscillatory integrals with uniformity in parameters.
		\newblock {\em Journal de th{\'e}orie des nombres de Bordeaux}, 31(1):145--159,
		2019.
		
		\bibitem{khan2023error}
		R.~Khan and Z.~Zhang.
		\newblock On the error term in a mixed moment of $L$-functions.
		\newblock {\em Mathematika}, 69(3):573--583,
		2023.
		
		\bibitem{kowalski2002rankin}
		E.~Kowalski, P.~Michel and J.~VanderKam.
		\newblock Rankin-Selberg $L$-functions in the level aspect.
		\newblock {\em Duke Mathematical Journal}, 114(1):123--191,
		2002
		
		\bibitem{lin2021analytic}
		Y.~Lin and Q.~Sun.
		\newblock Analytic twists of {$\rm{GL_3}\times \rm{GL_2}$} automorphic forms.
		\newblock {\em International Mathematics Research Notices}, 2021(19):15143--15208,
		2021.
		
		\bibitem{meurman1987order}
		T.~Meurman.
		\newblock On the order of the {M}aass {$L$}-function on the critical line, volume~51.
		\newblock {\em Number theory}, 325--354,
		1987.
		
		\bibitem{potter1940mean}
		H.~Potter.
		\newblock The mean values of certain {D}irichlet series, \uppercase\expandafter{\romannumeral 1}.
		\newblock {\em Proceedings of the London Mathematical Society}, 2(1):467--478,
		1940.
		
		\bibitem{shparlinski2019sums}
		I.~Shparlinski.
		\newblock On sums of Kloosterman and Gauss sums.
		\newblock {\em Transactions of the American Mathematical Society}, 371(12):8679--8697,
		2019.
		
		\bibitem{tang2025mixed}
		Z.~Tang and X.~Wu.
		\newblock Mixed moments of twisted $L$-functions.
		\newblock arXiv:2512.09203 [math.NT], 
		2025.
		
		\bibitem{MR1683661}
		N.~I. Zavorotny{\u{\i}}.
		\newblock On the fourth moment of the {R}iemann zeta function.
		\newblock {\em Automorphic functions and number theory, {P}art {\uppercase\expandafter{\romannumeral 1}}, {\uppercase\expandafter{\romannumeral 2}}
			({R}ussian)}, 69--124a, 254. Akad. Nauk SSSR, Dal' nevostochn.
		Otdel., Vladivostok, 1989.
		
	\end{thebibliography}

\end{document}